\documentclass[reqno]{llncs}

\usepackage{times,amsmath,amssymb,cmll}

\usepackage{enumitem}

\newcommand{\N}{\mathbb{N}}

\newcommand{\emptyWord}{\varepsilon}
\newcommand{\step}[1]{\to_{#1}}
\newcommand{\Powerset}[1]{2^{#1}}
\newcommand{\upclosure}[2]{{#2}{\uparrow_{#1}}}

\begin{document}

\title{Rational subsets and submonoids of wreath products  \thanks{This work was supported by the DAAD
research project RatGroup. The second author was partially supported by a grant from the Simons Foundation (\#245268 to Benjamin Steinberg).}}

\author{Markus Lohrey\inst{1} \and Benjamin Steinberg\inst{2} \and Georg Zetzsche\inst{3}}

\institute{Institut f\"ur Informatik, Universit\"at Leipzig,
  Germany\and
  City College of New York\and
  Technische Universit\"at Kaiserslautern \\ \email{zetzsche@cs.uni-kl.de}}

\maketitle

\begin{abstract}
It is shown that membership in rational subsets of wreath products
$H \wr V$ with $H$ a finite group and $V$ a virtually free group
is decidable. On the other hand, it is shown that there exists a
fixed finitely generated submonoid in the wreath product $\mathbb{Z}\wr\mathbb{Z}$ with
an undecidable membership problem.
\end{abstract}

\section{Introduction}

The study of algorithmic problems in group theory has a long tradition. Dehn, in his seminal paper
from 1911~\cite{Dehn11}, introduced the word problem (Does a given word over the generators
represent the identity?), the conjugacy problem (Are two given group elements conjugate?)
and the isomorphism problem
(Are two given finitely presented groups isomorphic?), see~\cite{LySch77} for
general references in combinatorial group theory.
Starting with the work of Novikov and Boone from the 1950's, all three problems were shown to be
undecidable for finitely presented groups in general.
A generalization of the word problem is the {\em subgroup membership
problem} (also known as the {\em generalized word problem}) for finitely generated groups: Given
group elements $g, g_1, \ldots, g_n$, does $g$ belong to the subgroup generated by $g_1, \ldots, g_n$?
Explicitly, this problem was introduced by Mihailova in 1958, although Nielsen had already presented
an algorithm for the subgroup membership problem for free groups in his paper from 1921~\cite{Nie21}.

Motivated partly by automata theory, the subgroup membership
problem was further generalized to {\em the rational subset membership problem}.
Assume that the group $G$ is finitely generated by the set $X$ (where $a \in X$ if and only if
$a^{-1} \in X$). A finite automaton $A$ with transitions labeled by
elements of $X$ defines a subset $L(A) \subseteq G$ in the natural
way; such subsets are the rational subsets of $G$. The rational subset membership
problem asks whether a given group element belongs to $L(A)$ for a
given finite automaton (in fact, this problem makes sense for any
finitely generated monoid). The notion of a rational subset in a monoid
can be traced back to the work of Eilenberg and Sch\"utzenberger from 1969~\cite{EiSchu69}.
Other early references are~\cite{Ani71,Gilman96}.
Rational subsets of groups also found applications for the solution of word equations
(here, quite often the term rational constraint is used)~\cite{DiMu06,LohSen06icalp}.
In automata theory, rational subsets are tightly related to valence automata:
For any group $G$, the emptiness problem for valence automata over $G$ (which are
also known as $G$-automata) is decidable if and only if $G$ has a decidable
rational subset membership problem. See~\cite{FeSt02,Kam09,KaSiSt06} for
details on valence automata and $G$-automata.

For free groups, Benois~\cite{Benois69} proved that the rational subset membership
problem is decidable using a classical automaton saturation procedure
(which yields a polynomial time algorithm). For commutative groups,
the rational subset membership can be solved using integer programming.
Further (un)decidability results on the rational subset membership problem
can be found in~\cite{LohSte08} for right-angled Artin groups,
in~\cite{Rom99} for nilpotent groups, and in~\cite{LohSt09tocs} for metabelian groups.
In general, groups with a decidable rational subset membership problem
seem to be rare. In~\cite{LohSt10} it was shown that if the group $G$ has
at least two ends, then the rational subset membership problem for $G$
is decidable if and only if the submonoid membership problem for $G$
(Does a given element of $G$ belong to a given finitely generated submonoid of $G$?)
is decidable.

In this paper, we investigate the rational subset membership problem for
wreath products. The wreath product is a fundamental operation in
group theory. To define the wreath product $H \wr G$ of
two groups $G$ and $H$, one first takes the direct sum
$K = \bigoplus_{g \in G} H$ of copies of $H$, one for each element of $G$.
An element $g \in G$ acts on $K$  by permuting the $G$-copies
of $H$  according to the left action of $g$ on $G$.
The corresponding semidirect product is the wreath product $H \wr G$.

In contrast to the word problem,
decidability of the rational subset membership problem is not preserved
under  wreath products. For instance, in~\cite{LohSt09tocs} it was shown that for every
nontrivial group $H$, the rational subset membership problem for
$H \wr (\mathbb{Z} \times \mathbb{Z})$ is undecidable.
The proof uses an encoding of a tiling problem, which uses the grid
structure of the Cayley graph of $\mathbb{Z} \times \mathbb{Z}$.

In this paper, we prove the following two new results concerning the rational subset membership problem
and the submonoid membership problem for
wreath products:
\begin{enumerate}[label=(\roman*)]
\item The submonoid membership problem is undecidable for
$\mathbb{Z} \wr \mathbb{Z}$. The wreath product $\mathbb{Z} \wr \mathbb{Z}$ is one of the simplest
examples of a finitely generated group that is not finitely presented, see \cite{DaOl11,Clea06} for further results showing
the importance of $\mathbb{Z} \wr \mathbb{Z}$.
\item For every finite group $H$ and every virtually free group\footnote{Recall that a group is virtually
free if it has a free subgroup of finite index.} $V$, the group
$H \wr V$ has a decidable rational subset membership problem; this includes
for instance the famous lamplighter group
$\mathbb{Z}_2 \wr \mathbb{Z}$.
\end{enumerate}
For the proof of (i) we encode the acceptance problem for a 2-counter machine
(Minsky machine~\cite{Min67}) into the submonoid membership problem for
$\mathbb{Z} \wr \mathbb{Z}$. One should remark that $\mathbb{Z} \wr
\mathbb{Z}$ is a finitely generated metabelian group and hence has
a decidable subgroup membership problem~\cite{Rom74,Rom80}.
For the proof of (ii), an automaton saturation procedure is used.  The
termination of the process is guaranteed by a well-quasi-order (wqo) which
refines the classical subsequence wqo considered by Higman~\cite{Hig52}.

Wqo theory has also been applied successfully
for the verification of
infinite state systems. This research led to the notion of
well-structured transition  systems~\cite{FinkelS01}. An application in formal
language theory is the decidability of the membership problem for
leftist grammars~\cite{MotwaniPSV00}.
Usually, a disadvantage of using wqo theory is that it does not yield
algorithms with good complexity bounds.
In the context of well-structured transition systems, several natural reachability problems (e.g. for lossy
channel systems) were shown to be not primitive recursive~\cite{ChambartS07,Schnoebelen02}. Also the  membership problem for
leftist grammars was shown be not primitive recursive~\cite{Jurdzinski08}.
The complexity status for the rational subset membership problem
for wreath products $H \wr V$ ($H$ finite, $V$ virtually free) remains
open. Actually, we do not even know whether the rational subset
membership problem for the lamplighter group  $\mathbb{Z}_2 \wr
\mathbb{Z}$ is primitive recursive.

As mentioned earlier, the rational subset membership problem is undecidable
for every wreath product $H \wr (\mathbb{Z} \times \mathbb{Z})$, where $H$ is a nontrivial group.
We conjecture that this can be generalized to the following result:
For every nontrivial group $H$ and every non-virtually free group
$G$, the rational subset membership problem for $H \wr G$ is undecidable.
The reason is that the undecidability proof for $H \wr (\mathbb{Z} \times \mathbb{Z})$
~\cite{LohSt09tocs} only uses the grid-like structure of the Cayley graph of
$\mathbb{Z} \times \mathbb{Z}$. In~\cite{KuLo04annals} it was shown that the Cayley graph of
a group $G$ has bounded tree width if and only if the group is virtually free.
Hence, if $G$ is not virtually free, then the Cayley-graph of $G$ has unbounded
tree width, which means that finite grids of arbitrary size appear as minors
in the Cayley-graph of $G$. One might therefore hope to reduce again a tiling
problem into the rational subset membership problem for $H \wr G$ (for $H$ non-trivial
and $G$ not virtually free).

Our decidability result for the rational subset membership problem for wreath products
$H \wr V$ with  $H$  finite  and $V$  virtually free can be also interpreted in terms
of tree automata with additional data values. Consider a tree walking automaton operating
on infinite rooted trees. Every tree node contains an additional data value from a finite group such that
all but finitely many nodes contain the group identity.
Besides navigating in the tree, the tree automaton  can multiply (on the right) the group element from
the current tree node with another group element (specified by the transition). The automaton cannot
read the group element from the current node.
Our decidability  result basically says that reachability
for this automaton model is decidable.

\section{Rational subsets of groups}

Let $G$ be a finitely generated group and $X$ a finite symmetric generating set for $G$
(symmetric means that $X$ is closed under taking inverses).
For a subset $B \subseteq G$ we denote with $B^*$ the {\em submonoid} of $G$
generated by $B$. The subgroup generated by $B$ is $\langle B \rangle$.
The set  of {\em rational subsets} of $G$ is the smallest set that (i) contains
all finite subsets of $G$ and (ii) that is closed under union, product, and $^*$.
Alternatively, rational subsets can be represented by finite automata.
Let $A = (Q,G,E,q_0,Q_F)$ be a finite automaton, where transitions are labeled with elements of $G$:
$Q$ is the finite set of states, $q_0 \in Q$ is the initial state, $Q_F \subseteq Q$ is the set of final
states, and $E \subseteq Q \times G \times Q$ is a finite set of transitions.
Every transition label $g \in G$ can be represented by a finite word over the generating
set $X$. In this way, $A$ becomes a finite object.
The subset $L(A) \subseteq G$ accepted by $A$ consists of all group elements
$g_1 g_2 g_3 \cdots g_n$ such that there exists a sequence of transitions
$(q_0, g_1, q_1), (q_1, g_2, q_2), (q_2, g_3, q_3), \ldots, (q_{n-1}, g_n,q_n) \in E$ with
$q_n \in Q_F$.
The {\em rational subset membership problem} for $G$ is the following decision problem:

\smallskip
\noindent
INPUT: A finite automaton $A$ as above and an element $g \in G$.

\noindent
QUESTION: Does $g \in L(A)$ hold?

\smallskip
\noindent
Since $g \in L(A)$ if and only if $1_G \in L(A)g^{-1}$, and
$L(A)g^{-1}$ is rational, too, the rational subset membership problem for $G$
is equivalent to the question of deciding whether a given automaton accepts the group identity.

The {\em submonoid membership problem} for $G$ is the following decision problem:

\smallskip
\noindent
INPUT: Elements $g, g_1, \ldots, g_n \in G$.

\noindent
QUESTION: Does $g \in \{g_1, \ldots, g_n\}^*$ hold?

\smallskip
\noindent
Clearly, decidability of the  rational subset membership problem for $G$
implies decidability of the submonoid membership problem for $G$. Moreover,
the latter generalizes the classical subgroup membership problem for $G$ (also
known as the generalized word problem), where the input is the same as for
the submonoid membership problem for $G$ but it is asked whether
$g \in \langle g_1, \ldots, g_n\rangle$ holds.

In our undecidability results in Section~\ref{sec-undec}, we will actually consider the non-uniform  variant
of the submonoid membership problem, where the submonoid is fixed, i.e., not part of the input.

\section{Wreath products}

Let $G$ and $H$ be groups.
Consider the direct sum
\[K = \bigoplus_{g \in G} H_g,\]
where $H_g$ is a copy
of $H$. We view $K$ as the set
$H^{(G)}=\{ f \in H^G \mid f^{-1}(H \setminus\{1_H\})\ \text{is finite}\}$ of all mappings
from $G$ to $H$ with finite support together with pointwise multiplication as
the group operation. The group $G$ has a natural left action on $H^{(G)}$ given by 
\[gf(a) = f(g^{-1}a)\]
where
$f \in H^{(G)}$ and $g, a \in G$.
The corresponding semidirect product $H^{(G)} \rtimes G$ is the wreath product $H \wr G$.
In other words:
\begin{itemize}
\item
Elements of $H \wr G$ are pairs $(f,g)$, where $g \in G$ and
$f \in H^{(G)}$.
\item
The multiplication in $H \wr G$ is defined as follows:
Let $(f_1,g_1), (f_2,g_2) \in H \wr G$. Then
$(f_1,g_1)(f_2,g_2) = (f, g_1g_2)$, where
$f(a) = f_1(a)f_2(g_1^{-1}a)$.
\end{itemize}
The following intuition might be helpful:
An element $(f,g) \in H\wr G$ can be thought of
as a finite multiset of elements of $H \setminus\{1_H\}$ that are sitting at certain
elements of $G$ (the mapping $f$) together with the distinguished
element $g \in G$, which can be thought of as a cursor
moving in $G$.
If we want to compute the product $(f_1,g_1) (f_2,g_2)$, we do this
as follows: First, we shift the finite collection of $H$-elements that
corresponds to the mapping $f_2$ by $g_1$: If the element $h \in H\setminus\{1_H\}$ is
sitting at $a \in G$ (i.e., $f_2(a)=h$), then we remove $h$ from $a$ and
put it to the new location $g_1a \in H$. This new collection
corresponds to the mapping $f'_2 \colon  a \mapsto f_2(g_1^{-1}a)$.
After this shift, we multiply the two collections of $H$-elements
pointwise: If in $a \in G$ the elements $h_1$ and $h_2$ are sitting
(i.e., $f_1(a)=h_1$ and $f'_2(a)=h_2$), then we put the product
$h_1h_2$ into the location $a$. Finally, the new distinguished
$G$-element (the new cursor position) becomes $g_1 g_2$.

\begin{proposition}\label{finiteindex}
Let $K$ be a subgroup of $G$ of finite index $m$ and let $H$ be a group.  Then $H^m\wr K$ is isomorphic to a subgroup of index $m$ in $H\wr G$.
\end{proposition}
\begin{proof}
Let $T$ be a set of right coset representatives for $G/K$; it has $m$ elements.
The action of $G$ on $H^{(G)}$ restricts to an action of $K$ on $H^{(G)}$ and so $H^{(G)}\rtimes K$ is a subgroup of $H\wr G$.
There is a $K$-equivariant\footnote{A $K$-equivariant group isomorphism $\alpha\colon H^{(G)}\to (H^T)^{(K)}$
is an isomorphism that commutes with the action of $K$: $k\alpha(f) = \alpha(kf)$.} group isomorphism $\alpha\colon H^{(G)}\to (H^T)^{(K)}$ given by $[\alpha(f)(k)](t) = f(kt)$,
where $f \in H^{(G)}$, $k \in K$, and $t \in T$.
This $\alpha$ is indeed bijective; the inverse $\alpha^{-1}$ is given by 
$[\alpha^{-1}(f)](kt) =[f(k)](t)$ for $f \in (H^T)^{(K)}$,
$k\in K$, and $t\in T$ (which has finite support because $T$ is finite and $f$ has finite support).
That $\alpha$ is $K$-equivariant follows from
\[[k\alpha(f)(k')](t)=[\alpha(f)(k^{-1}k')](t)=f(k^{-1}k't) = [kf](k't) = [\alpha(kf)(k')](t).\] 
It follows that $H^m\wr K\cong (H^T)^{(K)}\rtimes K\cong H^{(G)}\rtimes K$.

It thus remains to prove that $H^{(G)}\rtimes K$ has index $m$ in $H\wr G$.  Indeed,
let $e \in H^{(G)}$ be the map sending all of $G$ to the identity of $H$.  Then the elements of
the form $(e,t)$ with $t\in T$ form a set of right coset representatives of $H^{(G)}\rtimes K$ in $H\wr G$.  Indeed, it is easy to see that
these elements are in distinct cosets. If $g=kt$ with $k\in K$ and $t\in T$, then $(f,g) = (f,k)(e,t)$, which is in the coset of $(e,t)$.
\qed
\end{proof}

\section{Decidability} \label{sec-dec}

We show that the rational subset membership problem is decidable for groups
$G=H\wr V$, where $H$ is finite and $V$ is virtually free.
First, we will show that the rational subset membership problem for
$G=H\wr F_2$, where $F_2$ is the free group generated by $a$ and $b$,
is decidable. For this we make use of a particular well-quasi-order.

\subsection{A well-quasi-order} \label{sec-wqo}
Recall that a {\em well-quasi-order} on a set $A$ is a reflexive and transitive relation $\preceq$
such that for every infinite sequence $a_1, a_2, a_3,\ldots$ with $a_i \in A$ there exist
$i < j$ such that $a_i \preceq a_j$. In this paper, $\preceq$ will be always antisymmetric
as well; so $\preceq$ will be a well partial order.

For a finite alphabet $X$ and two words $u,v \in X^*$, we write
$u\preceq v$ if there exist $v_0,\ldots,v_n\in X^*$, $u_1,\ldots,u_n\in X$
such that $v=v_0u_1v_1\cdots u_nv_n$ and $u=u_1\cdots u_n$.
The following theorem was shown by Higman \cite{Hig52} (and independently Haines \cite{Hai69}).
\begin{theorem}[Higman's Lemma]
The order $\preceq$ on $X^*$ is a well-quasi-order.
\end{theorem}
Let $G$ be a group.
For a monoid morphism $\alpha\colon X^*\to G$ and $u,v \in X^*$
let $u\preceq_\alpha v$ if there is a factorization $v=v_0u_1v_1\cdots u_nv_n$ with
$v_0,\ldots,v_n\in X^*$, $u_1,\ldots,u_n\in X$, $u=u_1\cdots u_n$, and
$\alpha(v_i)=1$ for $0\le i\le n$. It is easy to see that $\preceq_\alpha$ is indeed
a partial order on $X^*$.
Furthermore, let $\preceq_G$ be the partial order on $X^*$ with $u\preceq_G v$ if
$v=v_0u_1v_1\cdots u_nv_n$ for some $v_0,\ldots,v_n\in X^*$, $u_1,\ldots,u_n\in
X$, and $u=u_1\cdots u_n$ such that $\alpha(v_i)=1$ for every morphism $\alpha\colon
X^*\to G$ and $0\le i\le n$. 
Note that if $G$ is finite, there are only finitely many morphisms
$\alpha\colon X^*\to G$. This means that for given $X$ and $G$, we can
construct a finite automaton for the upward closure $U\subseteq X^*$ of
$\{\emptyWord\}$ with respect to $\preceq_G$. Since for $w=w_1\cdots w_n$,
$w_1,\ldots,w_n\in X$, the upward closure of $w$ equals $Uw_1\cdots Uw_nU$, we
can also construct a finite automaton for the upward closure of any given
singleton provided that $G$ is finite.  In the latter case, we can also show
that $\preceq_G$ is a well-quasi-order:

\begin{lemma}
Let $G$ be a group. Then the following statements are equivalent:
\begin{enumerate}[label=(\roman*)]
\item\label{st:wqo} $(X^*, \preceq_G)$ is a well-quasi-order for each finite alphabet $X$.
\item\label{st:bounded} For every $n\in\N$, there is a $k\in\N$ with
$|\langle g_1,\ldots,g_n\rangle|\le k$ for all $g_1,\ldots,g_n\in G$.
\end{enumerate}
\end{lemma}

\begin{proof}
Suppose \ref{st:bounded} does not hold. Then there is a finite alphabet $X$ and a
sequence of morphisms $\alpha_1,\alpha_2,...\colon X^*\to G$ such that
$|\alpha_i(X^*)|\ge i$ for each $i\ge 1$. We inductively define a sequence of
words $w_1,w_2,\ldots\in X^*$. Choose $w_1=\emptyWord$ and suppose
$w_1,\ldots,w_i$ have been defined. Since $|\alpha_{i+1}(X^*)|\ge i+1$, we can
choose $w_{i+1}\in X^*$ to be a word such that $\alpha_{i+1}(w_{i+1})$ is
outside of $\{\alpha_{i+1}(w_1),\ldots,\alpha_{i+1}(w_i)\}$. We claim that the
words $w_1,w_2,\ldots$ are pairwise incomparable with respect to $\preceq_G$.
Observe that $u\preceq_G v$ implies $\alpha(u)=\alpha(v)$ for any morphism
$\alpha\colon X^*\to G$.  Since for any $i,j\in\N$, $i<j$, the construction
guarantees $\alpha_j(w_j)\ne\alpha_j(w_i)$, the words are pairwise
incomparable.

Suppose \ref{st:bounded} does hold and let $X$ be a finite alphabet. First, we claim
that there is a finite group $H$ such that $\preceq_G$ coincides with $\preceq_H$.
By \ref{st:bounded} there are only finitely many non-isomorphic groups
that appear as $\alpha(X^*)$ for morphisms $\alpha\colon X^*\to G$,
say $H_1,\ldots,H_m$, and each of them is finite. For $H=H_1\times\cdots\times H_m$,
we have
\[ \bigcap_{\alpha:X^*\to G} \ker(\alpha) = \bigcap_{\alpha:X^*\to H}\ker(\alpha).\]
Hence, $\preceq_G$ coincides with $\preceq_H$. There are only finitely many
morphisms $\alpha\colon X^*\to H$, say $\alpha_1, \ldots, \alpha_\ell$. If $\beta:
X^*\to H^\ell$ is the morphism with
$\beta(w)=(\alpha_1(w),\ldots,\alpha_\ell(w))$, then
\[ \bigcap_{\alpha:X^*\to H}\ker(\alpha)=\ker(\beta).\]
Thus, $\preceq_H$ coincides with $\preceq_\beta$. Therefore, it suffices to show
that $\preceq_\beta$ is a well-quasi-order.

Let $w_1,w_2,\ldots\in X^*$ be an infinite sequence of words. Since $H^\ell$ is
finite, we can assume that all the $w_i$ have the same image under $\beta$;
otherwise, choose an infinite subsequence on which $\beta$ is constant.  Consider the
alphabet $Y=X\times H^\ell$. For every $w\in X^*$, $w=a_1\cdots a_r$, let
$\bar{w}\in Y^*$ be the word
\begin{equation}\label{eq:embed}\bar{w}=(a_1, \beta(a_1))(a_2, \beta(a_1a_2))\cdots(a_r,\beta(a_1\cdots a_r)). \end{equation}
Applying Higman's Lemma to the sequence $\bar{w}_1,\bar{w}_2,\ldots$ yields
indices $i<j$ such that $\bar{w}_i\preceq \bar{w}_j$.  This means
$\bar{w}_i=u'_1\cdots u'_r$, $\bar{w}_j=v'_0u'_1v'_1\cdots u'_rv'_r$ for some
$u'_1,\ldots, u'_r\in Y$, $v'_0,\ldots,v'_r\in Y^*$. By definition of
$\bar{w}_i$ and $\bar{w}_j$, we have $u'_s=(u_s, h_s)$ for $1\le s\le r$, where
$h_s=\beta(u_1\cdots u_s)$ and $w_i=u_1\cdots u_r$. Let $\pi_1\colon Y^*\to X^*$ be
the morphism extending the projection onto the first component, and let
$v_s=\pi_1(v'_s)$ for $0\le s\le r$. Then clearly $w_j=v_0u_1v_1\cdots u_rv_r$.
We claim that $\beta(v_s)=1$ for $0\le s\le r$, from which $w_i\preceq_\beta w_j$
and hence the lemma follows.  Since $\bar{w}_j$ is also obtained
according to \eqref{eq:embed}, we have
\begin{equation*} \beta(u_1\cdots u_{s+1}) = h_{s+1} = \beta(v_0 u_1v_1 \cdots u_sv_s u_{s+1})\end{equation*}
for $0\le s\le r-1$. By induction on $s$, this allows us to deduce
$\beta(v_s)=1$ for $0\le s\le r-1$.  Finally, $\beta(w_i)=\beta(w_j)$ entails
\[\beta(u_1\cdots u_r)=\beta(w_i)=\beta(w_j)=\beta(v_0u_1v_1\cdots u_rv_r)=\beta(u_1\cdots u_r v_r),\]
implying $\beta(v_r)=1$.
\qed
\end{proof}

\subsection{Loops}
Let $G=H\wr F_2$ and fix free generators $a,b\in F_2$.  Recall that elements of $G$ are pairs $(k, f)$, where $k\in
K=\bigoplus_{g\in F_2} H$ and $f\in F_2$.  In the following, we simply
write $kf$ for the pair $(k,f)$.
Fix an automaton $A=(Q,G,E,q_0,Q_F)$  with labels from $G$ for the rest of Section~\ref{sec-dec}. 
We want to check whether $1 \in L(A)$. Since $G$ is generated as a monoid by $H\cup
\{a,a^{-1},b,b^{-1}\}$, we can assume that  $E\subseteq Q\times (H\cup
\{a,a^{-1},b,b^{-1}\})\times Q$.

A \emph{configuration} is an element of $Q\times G$. For
configurations $(p,g_1)$, $(q,g_2)$, we write $(p,g_1)\step{A}(q,g_2)$ if there
is a $(p,g,q)\in E$ such that $g_2=g_1g$. For elements $f,g\in F_2$, we write
$f\le g$ ($f<g$) if the reduced word representing $f$ is a (proper) prefix of
the reduced word representing $g$.
We say that an element $f\in F_2\setminus\{1\}$ is \emph{of type} $x \in \{a,a^{-1},b,b^{-1}\}$
if the reduced word representing $f$ ends with $x$.
Furthermore, $1\in F_2$ is of type $1$. Hence, the set of \emph{types} is
$T=\{1,a,a^{-1},b,b^{-1}\}$. When regarding the Cayley graph of $F_2$ as a
tree with root $1$, the children of a node of type $t$ are of the types
$C(t)=\{a,a^{-1},b,b^{-1}\}\setminus \{t^{-1}\}$. Clearly, two nodes have the same
type if and only if their induced subtrees of the Cayley graph are isomorphic.
The elements of $D=\{a,a^{-1},b,b^{-1}\}$ will also be called
\emph{directions}.

Let $p,q\in Q$ and $t\in T$. A sequence of configurations
\begin{equation} (q_1, k_1f_1) \step{A} (q_2, k_2f_2) \step{A} \cdots \step{A} (q_n, k_nf_n)\label{comp}\end{equation}
is called a \emph{well-nested $(p,q)$-computation in $t$}  if
\begin{enumerate}[label=(\roman*)]
\item $q_1=p$ and $q_n=q$,
\item $f_1=f_n$ is of type $t$, and
\item\label{cond:subtree} $f_i \ge f_1$ for $1< i < n$.
\end{enumerate}
We define the \emph{effect} of the computation to be $f_1^{-1}k_1^{-1}k_nf_1\in
K$.  Hence, the effect describes the change imposed by applying the
corresponding sequence of transitions, independently of the configuration in
which it starts. For $f\in F_2$, let $|f|$ be the length of the reduced word
representing $f$.  The \emph{depth} of the computation \eqref{comp} is the
maximum value of $|f_1^{-1}f_i|$ for $1\le i\le n$.  Of course, if $f_1=f_n=1$,
condition \ref{cond:subtree} is satisfied automatically. Hence, we have $1\in L(A)$
if and only if for some $q\in Q_F$, there is a well-nested
$(q_0,q)$-computation in $1$ with effect $1$.

For $d \in C(t)$,
a well-nested $(p,q)$-computation \eqref{comp} in $t$ is called a
\emph{$(p,d,q)$-loop in $t$} if in addition $f_1d \leq f_i$ for $1<i<n$.  
Note that  there is a $(p,d,q)$-loop in $t$ that starts in
$(p,kf)$  (where $f$ has type $t$)
with effect $e$ and depth $m$ if and only if there exists a
$(p,d,q)$-loop in $t$ with effect $e$ and depth $m$ that starts in $(p, t)$.

Given $p,q\in Q$, $t\in T$, $d\in C(t)$, it is decidable whether there is a
$(p,d,q)$-loop in $t$: This amounts to checking whether a given automaton with
input alphabet $\{a,a^{-1},b,b^{-1}\}$ accepts a word representing the identity
of $F_2$ such that no proper prefix represents the identity of $F_2$. Since this can be
accomplished using pushdown automata, we can compute the set
\[ X_t = \{(p,d,q)\in Q\times C(t)\times Q \mid \text{there is a
$(p,d,q)$-loop in $t$}\}. \]

\subsection{Loop patterns}
Given a word $w=(p_1,d_1,q_1)\cdots (p_n,d_n,q_n)\in X_t^*$, a \emph{loop
assignment for $w$} is a choice of a $(p_i,d_i,q_i)$-loop in $t$ for each position $i$,
$1\le i\le n$. The \emph{effect} of a loop assignment is $e_1\cdots e_n\in K$, where
$e_i\in K$ is the effect of the loop assigned to position $i$.  The
\emph{depth} of a loop assignment is the maximum depth of an appearing loop.  A
\emph{loop pattern for $t$} is a word $w\in X_t^*$ that has a loop assignment
with effect $1$.   The depth of the loop pattern is the minimum depth of a loop
assignment with effect $1$.  Note that applying the loops for the symbols in a
loop pattern $(p_1,d_1,q_1)\cdots (p_n,d_n,q_n)$ does not have to be a computation: We do not require
$q_i=p_{i+1}$.
Instead, the loop patterns describe the possible ways in which
a well-nested computation can enter (and leave) subtrees of the Cayley graph of
$F_2$ in order to have effect $1$.  The sets
\[ P_t = \{w\in X_t^* \mid \text{$w$ is a loop pattern for $t$} \} \]
for $t\in T$ will therefore play a central role in the decision procedure.

Recall the definition of the well-quasi-order $\preceq_H$ from Section~\ref{sec-wqo}.

\begin{lemma}\label{lemma:upwardclosed}
For each $t\in T$, the set $P_t$ is an upward closed subset of $X_t^*$ with respect to $\preceq_H$.
\end{lemma}

\begin{proof}
Since $K$ is a direct sum of copies of $H$, the orders $\preceq_H$ and $\preceq_K$
coincide.  It therefore suffices to show that $P_t$ is upward closed with
respect to $\preceq_K$.  Let $u\in P_t$ and $u \preceq_K v$, $v\in X_t^*$, meaning
$v=v_0u_1v_1\cdots u_nv_n$ with $u=u_1\cdots u_n$ and $\alpha(v_i)=1$, $0\le
i\le n$, for every morphism $\alpha\colon X_t^*\to K$.  Since $u\in P_t$, there is a
loop assignment for each $u_i$, $1\le i\le n$, with effect $e_i$ such that
$e_1\cdots e_n=1$. By construction of $X_t$, for each $(p,d,q)\in X_t$, there
is a $(p,d,q)$-loop, say $\ell_{p,d,q}$, in $t$. Let $\varphi\colon X_t^*\to K$ be the morphism such that
for each $(p,d,q)\in X_t$, $\varphi((p,d,q))$ is the effect of $\ell_{p,d,q}$.
Choosing $\ell_{p,d,q}$ for each occurrence of $(p,d,q)$ in a subword $v_i$
and reusing the loop assignments for the $u_i$ defines a loop assignment for $v$.
Since $\varphi(v_i)=1$ for $0\le i\le n$, the effect of this loop assignment is
$\varphi(v_0)e_1\varphi(v_1)\cdots e_n\varphi(v_n)=e_1\cdots e_n=1$. Hence, $v\in P_t$.
\qed
\end{proof}

Since $\preceq_H$ is a well-quasi-order, the previous lemma already implies
that each $P_t$ is a regular language.  On the one hand, this follows from the
fact that the upward closure of each singleton is regular.  On the other hand,
this can be deduced by observing that $\preceq_H$ is a monotone order in the
sense of \cite{EhHaRo83}.  Therein, Ehrenfeucht, Haussler, and Rozenberg show
that languages that are upward closed with respect to monotone
well-quasi-orders are regular.  Our next step is a characterization of the sets
$P_t$ that allows us to compute finite automata for them. In order to state
this characterization, we need the following definitions.

Let $X,Y$ be alphabets. A \emph{regular substitution} is a map $\sigma\colon X\to
\Powerset{Y^*}$ such that $\sigma(x)$ is a regular language for every $x\in X$.
For $w\in X^*$, $w=w_1\cdots w_n$, $w_i\in X$, let $\sigma(w)=R_1\cdots R_n$,
where $\sigma(w_i)=R_i$ for $1\le i\le n$.  Given a set $R\subseteq Y^*$ and a
regular substitution $\sigma\colon X\to \Powerset{Y^*}$, let $\sigma^{-1}(R)=\{w\in
X^*\mid \sigma(w)\cap R\ne\emptyset\}$. Note that if $R$ is regular, then
$\sigma^{-1}(R)$ is regular as well \cite[Proposition 2.16]{Sak09}, and an automaton for 
$\sigma^{-1}(R)$ can be constructed effectively from an automaton for
$R$ and automata for the $\sigma(x)$.\footnote{In
  \cite{Sak09}, it is shown that the class of regular languages is
  closed under arbitrary inverse substituitions. Moreover, the
  construction is effective if $\sigma$ is a regular substitution.}

The alphabet $Y_t$ is given by
\[Y_t = X_t \cup ( (Q\times H\times Q) \cap E).\]
The morphism $\pi_t\colon Y_t^*\to X_t^*$ is the projection onto $X_t^*$, meaning
\[\pi_t(y)= \begin{cases}
y & \text{ for } y\in X_t \\
\emptyWord & \text{ for } y\in Y_t\setminus X_t .
\end{cases}\]
The morphism $\nu_t\colon Y_t^*\to H$ is defined by
\begin{eqnarray*}
\nu_t((p,d,q)) & = & 1 \text{ for } (p,d,q)\in X_t \\
\nu_t((p,h,q)) & = & h \text{ for } (p,h,q)\in Y_t\setminus X_t.
\end{eqnarray*}
For $p,q\in Q$ and $t\in T$, define the regular set
\[R_{p,q}^t = \{ (p_0,g_1,p_1) ( p_1, g_2, p_2) \cdots (p_{n-1}, g_n,p_n) \in Y_t^* \mid p_0 = p, p_n = q \}.\]
Given $t\in T$
and $d\in C(t)$, the regular
substitution $\sigma_{t,d}\colon X_t\to\Powerset{Y_d^*}$ is defined by
\begin{eqnarray*}
\sigma_{t,d}((p,d,q)) &=& \bigcup \{ R^d_{p',q'} \mid (p,d,p'),  (q',d^{-1},q)\in E\} \\
\sigma_{t,d}((p,u,q)) &=& \{\varepsilon\} \ \text{ for } u\in C(t) \setminus\{d\}.
\end{eqnarray*}
Given two tuples, $(U_t)_{t\in T}$ and $(V_t)_{t\in T}$ with $U_t,V_t\subseteq
X_t^*$, we write $(U_t)_{t\in T}\le(V_t)_{t\in T}$ if $U_t\subseteq V_t$ for
each $t\in T$.

\begin{lemma}\label{lemma:smallest}
$(P_t)_{t\in T}$ is the smallest tuple such that for every $t\in T$ we have
$\emptyWord\in P_t$ and
\begin{equation} \bigcap_{d\in C(t)} \sigma_{t,d}^{-1}\left(\pi_d^{-1}(P_d) \cap \nu_d^{-1}(1)\right) \subseteq P_t. \label{eq:substclosure}\end{equation}
\end{lemma}

\begin{proof}
For each $i\in\N$, let $P^{(i)}_t\subseteq X_t^*$ be the set of loop patterns for $t$
whose depth is at most $i$.  Then $P^{(0)}_t=\{\emptyWord\}$ and
\[ P^{(i+1)}_t = P^{(i)}_t \cup \bigcap_{d\in C(t)} \sigma_{t,d}^{-1}\left(\pi_d^{-1}(P^{(i)}_d) \cap \nu_d^{-1}(1)\right).\]
The lemma follows since $P_t=\bigcup_{i\ge 0} P_t^{(i)}$.
\qed
\end{proof}

Given a language $L\subseteq X_t^*$, let
$\upclosure{t}{L}=\{v\in X_t^* \mid \text{$u\preceq_H v$ for some $u\in L$}\}$.

\begin{theorem}\label{maindecidablethm}
The rational subset membership problem
 is decidable for every group $G=H\wr F$,
where $H$ is finite and $F$ is a finitely generated free group.
\end{theorem}

\begin{proof}
Since $H\wr F$ is a subgroup of $H\wr F_2$ (since $F$ is a subgroup of $F_2$), it suffices to show decidability
for $G=H\wr F_2$.
First, we compute finite automata for the languages $P_t$.  We do this by
initializing $U_t^{(0)}:=\upclosure{t}{\{\emptyWord\}}$ for each $t\in T$ and then
successively extending the sets $U^{(i)}_t$, which are represented by finite
automata, until they equal $P_t$: If there is a $t\in T$ and a
word
\[ w\in \bigcap_{d\in C(t)} \sigma_{t,d}^{-1} \left(\pi_d^{-1}(U_d^{(i)})\cap\nu_d^{-1}(1)\right) \setminus U^{(i)}_t, \]
we set $U_t^{(i+1)}:=U^{(i)}_t \cup \upclosure{t}{\{w\}}$ and
$U^{(i+1)}_{u}:=U^{(i)}_{u}$ for $u\in T\setminus \{t\}$. Otherwise we stop.
By induction on $i$, it follows from Lemma \ref{lemma:upwardclosed} and Lemma
\ref{lemma:smallest} that $U^{(i)}_t\subseteq P_t$.

In each step, we obtain $U^{(i+1)}_t$ by adding new words to $U^{(i)}_t$.
Since the sets $U^{(i)}_t$ are upward closed by construction and there is no
infinite (strictly) ascending chain of upward closed sets in a wqo, the algorithm above
has to terminate with some tuple $(U^{(k)}_t)_{t\in T}$. This, however, means
that for every $t\in T$
\[ \bigcap_{d\in C(t)} \sigma_{t,d}^{-1}\left(\pi_d^{-1}(U^{(k)}_d)\cap \nu_d^{-1}(1)\right)\subseteq U^{(k)}_t. \]
Since on the other hand $\emptyWord\in U^{(k)}_t$ and $U^{(k)}_t\subseteq P_t$,
Lemma \ref{lemma:smallest} yields $U^{(k)}_t=P_t$.

Now we have $1\in L(A)$ if and only if $\pi_1^{-1}(P_1)\cap\nu_1^{-1}(1)\cap
R^1_{q_0,q}\ne\emptyset$ for some $q\in Q_F$, which can again be checked by
constructing and analyzing finite automata.
\qed
\end{proof}

\begin{theorem}
The rational subset membership problem
 is decidable for every group $H\wr V$ with $H$ finite and $V$ virtually free.
\end{theorem}
\begin{proof}
This is immediate from Theorem~\ref{maindecidablethm} and Proposition~\ref{finiteindex},
because if $F$ is a free subgroup of index $m$ in $V$, then $H^m\wr F$ is isomorphic to a subgroup of index $m$ in $H\wr V$
and decidability of rational subset membership is preserved by finite extensions~\cite{Gru99,KaSiSt06}.
\qed
\end{proof}

\section{Undecidability} \label{sec-undec}

In this section, we will prove the second main result of this paper:
The wreath product $\mathbb{Z}\wr\mathbb{Z}$ contains a fixed
submonoid with an undecidable membership problem. Our proof
is based on the halting problem for 2-counter machines (also known
as Minsky machines), which is a classical undecidable problem.

\subsection{2-counter machines}

A  2-counter machine (also known as
Minsky machine) is a tuple $C = (Q, q_0, q_f, \delta)$, where
\begin{itemize}
\item $Q$ is a finite set of states,
\item $q_0 \in Q$ is the initial state,
\item $q_f \in Q$ is the final state, and
\item $\delta \subseteq (Q \setminus \{q_f\}) \times \{c_0,c_1\} \times \{+1,-1,=0\} \times
  Q$ is the set of transitions.
\end{itemize}
The set of configurations is $Q \times \mathbb{N} \times \mathbb{N}$.
On this set we define a binary relation $\to_C$ as follows:
$(p, m_0, m_1) \to_C (q, n_0, n_1)$ if and only if one of the
following three cases holds:
 \begin{itemize}
\item There exist $i \in \{0,1\}$ and a
transition $(p, c_i, +1, q) \in \delta$ such that $n_i = m_i+1$ and
$n_{1-i}=m_{1-i}$.
\item There exist $i \in \{0,1\}$ and a
transition $(p, c_i, -1, q) \in \delta$ such that $n_i = m_i-1$ (in particular,
we must have $m_i > 0$) and
$n_{1-i}=m_{1-i}$.
\item There exist $i \in \{0,1\}$ and a
transition $(p, c_i, =0, q) \in \delta$ such that $n_i=m_i=0$ and
$n_{1-i}=m_{1-i}$.
\end{itemize}
It is well known that every Turing-machine can be simulated by a 2-counter machine,
see e.g.~\cite{Min67}. In particular, we have:

\begin{theorem} \label{thm-2-counter}
There exists a fixed  2-counter machine $C= (Q, q_0, q_f, \delta)$ such that the following
problem is undecidable:

\smallskip
\noindent
INPUT: Numbers $m,n \in \mathbb{N}$.

\noindent
QUESTION: Does $(q_0,m,n)  \to^*_C (q_f,0,0)$ hold?
\end{theorem}

\subsection{Submonoids of $\mathbb{Z} \wr \mathbb{Z}$}

 In this section, we will only consider wreath products of the form
 $H\wr\mathbb{Z}$.
An element $(f,m) \in H\wr\mathbb{Z}$ such that the support of $f$ is contained
in the interval $[a,b]$ (with $a,b \in \mathbb{Z}$)
and $0,m \in [a,b]$
will also be written as a list
$[f(a),\ldots,  f(b)]$,
where in addition the element $f(0)$ is labeled by an incoming (downward) arrow and the element $f(m)$ is labeled
by an outgoing (upward) arrow.

In this section,
we will construct a fixed finitely generated submonoid of the wreath product
$\mathbb{Z} \wr \mathbb{Z}$ with an undecidable membership problem.

Let $C= (Q, q_0, q_f, \delta)$ be the  2-counter machine from
Theorem~\ref{thm-2-counter}. Without loss of generality we can assume that
there exists a partition $Q = Q_0 \cup Q_1$ such that $q_0 \in Q_0$ and
\[\delta \subseteq (Q_0 \times \{c_0\} \times \{+1,-1,=0\} \times Q_1) \cup
(Q_1 \times \{c_1\} \times \{+1,-1,=0\} \times Q_0).\]
In other words, $C$ alternates between the two
counters.  Hence, a transition $(q,c_i, x, p)$ can be just written as $(q,x,p)$.
Let
\[\Sigma = Q \cup\{c,\#\}.\]
Let $\mathbb{Z}^{\Sigma}$ be
the free abelian group generated by $\Sigma$.
First, we will prove that there is a fixed finitely generated submonoid $M$ of the wreath product
$\mathbb{Z}^{\Sigma} \wr \mathbb{Z}$  with an undecidable
membership problem. Let
\[K = \bigoplus_{m \in \mathbb{Z}} \mathbb{Z}^{\Sigma}.\]
An infinite presentation for $\mathbb{Z}^{\Sigma} \wr \mathbb{Z}$ is
\[\mathbb{Z}^{\Sigma} \wr \mathbb{Z} = \langle \Sigma,a \mid [a^n s
a^{-n}, a^m t a^{-m}] = 1\; (n,m \in \mathbb{Z}, s,t \in \Sigma)\rangle.\]
Here, $a$ generates the right $\mathbb{Z}$-factor of the wreath product
$\mathbb{Z}^{\Sigma} \wr \mathbb{Z}$.
In the following, we will freely switch between the description of elements of $\mathbb{Z}^{\Sigma} \wr \mathbb{Z}$
by words over $(\Sigma \cup \{a\})^{\pm 1}$ and by pairs from
$K \rtimes \mathbb{Z}$.
For a finite-support mapping $f \in K$,
$m \in \mathbb{Z}$, and $x \in \Sigma$, we also write $f(m,x)$ for the integer $f(m)(x)$.

Our  finitely generated submonoid $M$ of $\mathbb{Z}^{\Sigma} \wr
\mathbb{Z}$ is generated by the following elements. The right
column shows the generators in list notation, where elements of the free abelian group $\mathbb{Z}^{\Sigma}$
are written additively, i.e., as $\mathbb{Z}$-linear combinations of elements of $\Sigma$:
\begin{alignat}{2}
  & p^{-1} a \# a^2 \# a q \ \text{for} \ (p,=0,q) \in \delta  &&   [\overset{\shpos}{-p}, \#, 0, \#, \overset{\shneg}{q}] \label{generator=0} \\
  & p^{-1} a \# a c  a^2 q a^{-2} \ \text{for} \ (p,+1,q) \in \delta && [\overset{\shpos}{-p}, \#,  \overset{\shneg}{c}, 0, q] \label{generator+1} \\
  & p^{-1} a \# a^3 q a^6  c^{-1} a^{-8} \ \text{for} \ (p,-1,q) \in  \delta  \qquad && [\overset{\shpos}{-p}, \#,  \overset{\shneg}{0}, 0,q,0,0,0,0,0,-c] \label{generator-1} \\
  & c^{-1}  a^8 c a^{-8}    && [\overset{\shpos\shneg}{-c}, 0, 0, 0, 0, 0, 0, 0, c] \label{generator-copy1} \\
  & c^{-1}  a \# a^7 c a^{-6}  && [\overset{\shpos}{-c}, \#, \overset{\shneg}{0}, 0, 0, 0, 0, 0, c] \label{generator-copy2}  \\
  & q_f^{-1} a^{-1}  && [\overset{\shneg}{0},\overset{\shpos}{-q_f}] \label{generator-terminate} \\
  & \#^{-1} a^{-2} && [\overset{\shneg}{0},0,\overset{\shpos}{-\#}]\label{generator-go-left}
\end{alignat}
For initial counter values $m,n \in \mathbb{N}$ let
\[I(m,n) =  a q_0  a^2  c^m a^4  c^n a^{-6}.\]
The list notation for $I(m,n)$ is
\begin{equation} \label{list-initial}
[\overset{\shpos}{0}, \overset{\shneg}{q_0}, 0, m \cdot c, 0, 0,0,n \cdot c] .
\end{equation}
Here is some intuition:
The group element $I(m,n)$ represents the initial configuration $(q_0,m,n)$ of  the 2-counter machine
$C$. Lemma~\ref{lemma-main-undec} below states that $(q_0,m,n)  \to^*_C (q_f,0,0)$ is equivalent to the existence
of $Y \in M$ with $I(m,n) Y = 1$, i.e., $I(m,n)^{-1} \in M$.
Generators of type \eqref{generator=0}--\eqref{generator-copy2} simulate the 2-counter machine
$C$. States of $C$ will be stored at cursor positions $4k+1$. The values of the first (resp., second) counter will
be stored at cursor positions $8k+3$ (resp., $8k+7$).
Note that $I(m,n)$ puts a single copy of the symbol $q_0 \in  \Sigma$
at position $1$, $m$ copies of symbol $c$ (which represents counter
values) at position 3, and $n$ copies of symbol $c$ at position 7.
Hence, indeed, $I(m,n)$ sets up the initial configuration $(q_0,m,n)$
for $C$. Even cursor positions will carry the special
symbol $\#$. Note that generator \eqref{generator-terminate} is the only generator which changes the
cursor position from even to odd or vice versa. It will turn out that if $I(m,n) Y = 1$ ($Y \in M$),
then generator \eqref{generator-terminate} has to occur exactly once in $Y$; it terminates the simulation
of the 2-counter machine $C$. Hence, $Y$ can be written as $Y = U (q_f^{-1} a^{-1}) V$ with $U,V \in M$.
Moreover, it turns out that $U \in M$ is a product of generators \eqref{generator=0}--\eqref{generator-copy2},
which simulate $C$. Thereby, even cursor positions will be marked with
a single occurrence of the special symbol $\#$.
In a second phase, that corresponds to $V \in M$, these special symbols $\#$ will be removed again
and the cursor will be moved left to position $0$. This is accomplished with generator \eqref{generator-go-left}.
In fact, our construction enforces that $V$ is a power of \eqref{generator-go-left}.

During the simulation phase (corresponding to $U \in M$), generators of type \eqref{generator=0}
implement zero tests, whereas generators of type \eqref{generator+1} (resp., \eqref{generator-1})
increment (resp., decrement) a counter.
Finally, \eqref{generator-copy1} and \eqref{generator-copy2} copy the counter value to the next
cursor position that is reserved for the counter (that is copied). During such a copy phase,
\eqref{generator-copy1} is first applied $\geq 0$ many times. Finally,    \eqref{generator-copy2}
is applied exactly once.

\begin{lemma}  \label{lemma-main-undec}
For all $m,n \in \mathbb{N}$ the following are equivalent:
\begin{itemize}
\item $(q_0,m,n)  \to^*_C (q_f,0,0)$
\item There exists $Y \in M$ such that $I(m,n) Y = 1$.
\end{itemize}
\end{lemma}

\begin{proof}
Assume first that $I(m,n) Y = 1$ for some $Y \in M$.
We have to show that $(q_0,m,n)  \to^*_C (q_f,0,0)$; this is the more
difficult direction.
Let
\[
Y = y_1 \cdots y_k,
\]
where each $y_i$ is one of the generators of $M$.
For $0 \leq i \leq k$ let
\[
 Y_i =  y_1 \cdots y_i
\]
(thus, $Y_0 = 1$)
and assume that
\[I(m,n) Y_i = (f_i,m_i) \in K \rtimes \mathbb{Z}.\]
Hence, $f_k = 0$ is the zero-mapping and $m_k=0$.
Moreover $(f_0,m_0) = I(m,n)$.

\medskip
\noindent
{\em Claim 1.} For all $0 \leq i \leq k$, $q \in Q$, and $\ell \in \mathbb{Z}$
we have
$f_i(2\ell,q) = 0$.

\medskip
\noindent
{\em Proof of Claim 1.}  Assume that $f_i(2\ell,q) \neq 0$ for some $0 \leq i \leq k$, $q \in Q$, and $\ell \in \mathbb{Z}$.
Choose $0 \leq i \leq k$ minimal such that there exist $q \in Q$ and $\ell \in \mathbb{Z}$
with $f_i(2\ell,q) \neq 0$. Since $f_0(2\ell,q)=0$ for all $q \in Q$ and $\ell \in \mathbb{Z}$ (the list notation for 
$(f_0,m_0)$ is \eqref{list-initial}),
we must have $i \geq 1$. Hence,
$f_{i-1}(2\ell,q) = 0$  for all $q \in Q$ and $\ell \in \mathbb{Z}$.
An inspection of the generators shows that
if $m_{i-1}$ were odd, we would also have $f_i(2\ell,q) = 0$  for all $q \in Q$ and $\ell \in \mathbb{Z}$.
Therefore, $m_{i-1}$ must be even.
An inspection of the generators
of $M$ shows that there exist $j \in \mathbb{Z}$ and $p \in Q$ such that
\[
f_i(2j,p) < 0 \text{ and } f_i(2j',p') = 0 \text{ for all $j' < j$ and $p' \in Q$.}
\]
But then, for all $i \leq i' \leq k$ there exist
$j \in \mathbb{Z}$ and $p \in Q$ such that
\[
f_{i'}(2j,p) < 0 \text{ and } f_{i'}(2j',p') = 0 \text{ for all $j' < j$ and $p' \in Q$.}
\]
For $i'=k$ we obtain a contradiction, since $f_k = 0$.

\medskip
\noindent
Claim 1 implies that for all $1 \leq i \leq k$ with $m_{i-1}$ even, the generator
$y_i$ cannot be of type \eqref{generator=0}, \eqref{generator+1}, \eqref{generator-1}, or
\eqref{generator-terminate}.

\medskip
\noindent
{\em Claim 2.} For all $0 \leq i \leq k$ and $\ell \in \mathbb{Z}$
we have
$f_i(2\ell,c) = 0$.

\medskip
\noindent
{\em Proof of Claim 2.}  Assume that $f_i(2\ell,c) \neq 0$ for some $0 \leq i \leq k$ and $\ell \in \mathbb{Z}$.
Choose $0 \leq i \leq k$ minimal such that there exists  $\ell \in \mathbb{Z}$
with $f_i(2\ell,c) \neq 0$. Since $f_0(2\ell,c)=0$ for all $\ell \in \mathbb{Z}$,
we must have $i \geq 1$.  Hence,
$f_{i-1}(2\ell,c) = 0$  for all $\ell \in \mathbb{Z}$.
An inspection of the generators shows that
if $m_{i-1}$ were odd, we would also have $f_i(2\ell,c) = 0$  for all $\ell \in \mathbb{Z}$.
Therefore,  $m_{i-1}$ must be even.
The generator $y_i$ must be of one of the types
\eqref{generator+1}, \eqref{generator-1},  \eqref{generator-copy1}, or \eqref{generator-copy2}.
But the types   \eqref{generator+1} and \eqref{generator-1} are excluded by the remark
before Claim~2.
Therefore, $y_i$ must be either \eqref{generator-copy1} or \eqref{generator-copy2}.
Thus,  there exists $j \in \mathbb{Z}$  such that
\[
f_i(2j,c) < 0 \text{ and } f_i(2j',c) = 0 \text{ for all $j' < j$.}
\]
Note that for all $i < i' \leq k$ with $m_{i'-1}$ even, the generator $y_{i'}$
is not of  type  \eqref{generator+1} (again by the remark before Claim~2).
This implies that for all $i \leq i' \leq k$ there exists
$j \in \mathbb{Z}$  such that
\[
f_{i'}(2j,c) < 0 \text{ and } f_{i'}(2j',c) = 0 \text{ for all $j' < j$.}
\]
For $i'=k$ we obtain a contradiction, since $f_k = 0$.

\medskip
\noindent
Claim 1 and 2 imply that for all $1 \leq i \leq k$ with $m_{i-1}$ even, the generator
$y_i$ is 
\eqref{generator-go-left}.

\medskip
\noindent
{\em Claim 3.} For all $0 \leq i \leq k$ and $\ell \in \mathbb{Z}$
we have
$f_i(2\ell+1,\#) = 0$.

\medskip
\noindent
{\em Proof of Claim 3.}  Assume that $f_i(2\ell+1,\#) \neq 0$ for some $0 \leq i \leq k$ and $\ell \in \mathbb{Z}$.
Choose $0 \leq i \leq k$ minimal such that there exists  $\ell \in \mathbb{Z}$
with $f_i(2\ell+1,\#) \neq 0$. Since $f_0(\ell,\#)=0$ for all $\ell \in \mathbb{Z}$,
we must have $i \geq 1$.  Hence,
$f_{i-1}(2\ell+1,\#) = 0$  for all $\ell \in \mathbb{Z}$.
There are two possible cases:
\begin{enumerate}
\item $m_{i-1}$ is odd and $y_i$ is the generator \eqref{generator-go-left}.
\item $m_{i-1}$ is even and $y_i$ is a generator of type
  \eqref{generator=0}--\eqref{generator-1}
 or \eqref{generator-copy2}.
\end{enumerate}
But the second case is not possible by the remark before Claim 3.
Hence, $m_{i-1}$ is odd and $y_i$ is the generator \eqref{generator-go-left}.
Thus, there exists $j \in \mathbb{Z}$ with $f_i(2j+1,\#) < 0$.
Since for every $i \leq i' \leq k$ with $m_{i'-1}$ even, the generator $y_{i'}$
can only be of type  \eqref{generator-go-left}
(again by the remark before Claim~3), it follows
that for every $i \leq i' \leq k$ we have
$f_{i'}(2j+1,\#) < 0$.
For $i'=k$ we obtain a contradiction, since $f_k = 0$.

\medskip
\noindent
{\em Claim 4.} There is exactly one $1 \leq i \leq k$ such that $y_i$ is the generator
\eqref{generator-terminate}.

\medskip
\noindent
{\em Proof of Claim 4.}
For $g = (f,m) \in  \mathbb{Z}^{\Sigma} \wr \mathbb{Z}$  and $b \in \{0,1\}$ we define
$$
\sigma_Q(g,b) = \sum_{k \in \mathbb{Z}} \sum_{q \in Q} f(2k+b,q) .
$$
An inspection of all generators of $M$ shows that for every $g \in  \mathbb{Z}^{\Sigma} \wr \mathbb{Z}$ and
every generator $z$ of $M$  we have:
\begin{itemize}
\item
If $z$ is not the generator \eqref{generator-terminate}, then
$\sigma_Q(gz,b) =  \sigma_Q(g,b)$ for both $b=0$ and $b=1$.
\item
If $z$ is the generator \eqref{generator-terminate}, then there is $b \in \{0,1\}$
such that
$\sigma_Q(gz,b) =  \sigma_Q(g,b)-1$ and
$\sigma_Q(gz,1-b) =  \sigma_Q(g,1-b)$.
\end{itemize}
The claim follows, since $\sigma_Q(I(m,n),0)= \sigma_Q(I(m,n)Y,0)=
\sigma_Q(I(m,n)Y,1) = 0$ and $\sigma_Q(I(m,n),1)=1$.

\medskip
\noindent
By Claim 1--4, there exists a unique $1 \leq i \leq k$ such that
the following three properties hold:
\begin{itemize}
\item For every $1 \leq j < i$, $y_j$ is a generator of type \eqref{generator=0}--\eqref{generator-copy2}.
\item $y_i$ is the generator \eqref{generator-terminate}.
\item For every $i < j \leq k$, $y_j$ is the generator \eqref{generator-go-left}.
\end{itemize}
Hence, $I(m,n) Y_{i-1}$ must be of the form
$$
[\overset{\shpos}{0},0,\#,0,\#,0,\#, \ldots,0,\#,0,\#,\overset{\shneg}{q_f}],
$$
since only such an element can be reduced to 1 by right-multiplication
with generator \eqref{generator-terminate} followed by a positive
power of generator \eqref{generator-go-left}.
We show that this implies $(q_0,m,n)  \to^*_C (q_f,0,0)$.
Note that every generator of type  \eqref{generator=0}--\eqref{generator-copy2}
(those generators that occur in $Y_{i-1}$) moves the cursor $2d$ (for some $d \geq 0$)
to the right along the $\mathbb{Z}$-line.
This means that for every $0 \leq j \leq i-1$, $m_j$ is odd and moreover,
for every odd $m < m_j$, the group element $f_j(m) \in \mathbb{Z}^\Sigma$
is zero.

\medskip
\noindent
{\em Claim 5.}
Let $0 \leq j < i-1$  and assume that $I(m,n)Y_j$ is of the form
\begin{equation} \label{config-8k+1}
[\overset{\shpos}{0},0,\#,0,\#,0,\#,
\ldots,0,\#,0,\#,\overset{\shneg}{p},0,a \cdot c,0,0,0,b \cdot c],
\end{equation}
where
$p \in Q_0$, $a,b \in \mathbb{N}$, and $\overset{\shneg}{p}$ occurs
at position $\ell = 8k+1$ for some $k \geq 0$
(hence, \eqref{config-8k+1} represents the configuration $(p,a,b)$).
Then there exists $j' > j$ and a valid $C$-transition
$(p,a,b) \to_C (q,a',b')$ such that
$I(m,n)Y_{j'}$ is of the form
$$
[\overset{\shpos}{0},0,\#,0,\#,0,\#,
\ldots,0,\#,0,\#,\overset{\shneg}{q},0,b' \cdot c,0,0,0,a' \cdot c].
$$
Here $\overset{\shneg}{q}$ occurs at position $\ell+4$.

\medskip
\noindent
{\em Proof of Claim 5.}
Generator $y_{j+1}$ has to be of the form \eqref{generator=0},
\eqref{generator+1}, or  \eqref{generator-1}, because otherwise
we leave at position $\ell$ a negative copy of 
$c$, which cannot be compensated later.
Let us first assume that $y_{j+1}$ has the form \eqref{generator=0},
i.e., $(p,=0,q) \in \delta$. Then
$I(m,n)Y_{j+1}$ is of the form
\begin{equation} \label{claim-5-zero-test}
[\overset{\shpos}{0},0,\#,0,\#,0,\#,
\ldots,0,\#,a \cdot c,\#,\overset{\shneg}{q},0,b \cdot c,0,0,0,0],
\end{equation}
where $\overset{\shneg}{q}$ occurs at position $\ell+4$.
If $a > 0$, then the $a$ many $c$'s at position $\ell+2$
cannot be removed in the future. Hence, we must have
$a=0$.
Setting $a'=0$ and $b'=b$ shows that
\eqref{claim-5-zero-test} has the form required in the conclusion of Claim 5.

Next, assume that $y_{j+1}$ has the form \eqref{generator+1}. Hence
$(p,+1,q) \in \delta$ and
$I(m,n)Y_{j+1}$ is of the form
$$
[\overset{\shpos}{0},0,\#,0,\#,0,\#,
\ldots,0,\#,\overset{\shneg}{(a+1) \cdot c},0,q,0,b \cdot c,0,0,0,0],
$$
where $\overset{\shneg}{(a+1) \cdot c}$ occurs at position $\ell+2$.
So we have to remove $a+1$ many copies of $c$ from position $\ell+2$.
Hence, the only way to continue is to apply $a$ many times generator
\eqref{generator-copy1} followed by a single application of generator \eqref{generator-copy2}.
Hence, $I(m,n)Y_{j+a+2}$ must be of the form
\begin{equation} \label{claim-5-+1}
[\overset{\shpos}{0},0,\#,0,\#,0,\#,
\ldots,0,\#,0,\#,\overset{\shneg}{q},0,b \cdot c,0,0,0,(a+1)\cdot c],
\end{equation}
where $\overset{\shneg}{q}$ occurs at position $\ell+4$.
Setting $b'= b$ and $a'=a+1$ shows that
\eqref{claim-5-+1} has the form required in the conclusion of Claim 5.

Finally, assume that $y_{j+1}$ has the form \eqref{generator-1}, hence
$(p,-1,q) \in \delta$ and
$I(m,n)Y_{j+1}$ is of the form
$$
[\overset{\shpos}{0},0,\#,0,\#,0,\#,
\ldots,0,\#,\overset{\shneg}{a \cdot c},0,q,0,b \cdot c,0,0,0,-c],
$$
where $\overset{\shneg}{a \cdot c}$ occurs at position $\ell+2$.
First, assume that $a=0$.
Then there is no way to move the cursor to the right without leaving
a negative copy of a symbol from $Q \cup \{c\}$ at position $\ell+2$,
and this negative copy cannot be eliminated later. Hence, we must have
$a > 0$. Now, the only way to continue is to apply $a-1$ many times generator
\eqref{generator-copy1} followed by a single application of generator \eqref{generator-copy2}.
Hence, $I(m,n)Y_{j+a+1}$ must be of the form
\begin{equation} \label{claim-5--1}
[\overset{\shpos}{0},0,\#,0,\#,0,\#,
\ldots,0,\#,0,\#,\overset{\shneg}{q},0,b \cdot c,0,0,0,(a-1)\cdot c],
\end{equation}
where $\overset{\shneg}{q}$ occurs at position $\ell+4$.
Setting $b'= b$ and $a'=a-1$ shows that
\eqref{claim-5--1} has the form required in the conclusion of Claim 5.

This concludes the proof of Claim~5.
Completely analogously to Claim~5, one can show:

\medskip
\noindent
{\em Claim 6.}
Let $0 \leq j < i-1$  and assume that $I(m,n)Y_j$ is of the form
\begin{equation} \label{config-8k+5}
[\overset{\shpos}{0},0,\#,0,\#,0,\#,
\ldots,0,\#,0,\#,\overset{\shneg}{p},0,a \cdot c,0,0,0,b \cdot c],
\end{equation}
where
$p \in Q_1$, $a,b \in \mathbb{N}$, $\overset{\shneg}{p}$ occurs
at position $\ell = 8k+5$ for some $k \geq 0$
(hence, \eqref{config-8k+5} represents the configuration $(p,b,a)$).
Then there exists $j' > j$ and a valid $C$-transition
$(p,b,a) \to_C (q,b',a')$ such that
$I(m,n)Y_{j'}$ is of the form
$$
[\overset{\shpos}{0},0,\#,0,\#,0,\#,
\ldots,0,\#,0,\#,\overset{\shneg}{q},0,b' \cdot c,0,0,0,a' \cdot c].
$$
Here $\overset{\shneg}{q}$ occurs at position $\ell+4$.

\medskip
\noindent
Using Claim 5 and 6 we can now easily conclude that $(q_0,m,n) \to^*_C
(q_f,0,0)$ holds.

The other direction (if $(q_0,m,n)  \to^*_C (q_f,0,0)$ then there
exists $Y \in M$ with $I(m,n) Y = 1$) is easier.  A computation
$$
(q_0,m,n)  \to_C  (q_1,m_1,n_1) \to_C \cdots \to_C (q_{\ell-1},m_{\ell-1},n_{\ell-1}) \to_C  (q_f,0,0)
$$
can be directly translated into a sequence of $M$-generators $y_1 y_2 \cdots y_k$ such that the group element
$I(m,n) y_1 y_2 \cdots y_k$ has the form
$$
[\overset{\shpos}{0},0,\#,0,\#,0,\#, \ldots,0,\#,0,\#,\overset{\shneg}{q_f}],
$$
Multiplying this element with generator \eqref{generator-terminate} followed by a positive power
of generator \eqref{generator-go-left} yields the group identity.
\qed
\end{proof}
The following result is an immediate consequence of Theorem~\ref{thm-2-counter}
and Lemma~\ref{lemma-main-undec}.

\begin{theorem} \label{thm:undec1}
There is a
fixed finitely generated submonoid $M$ of the wreath product
$\mathbb{Z}^{\Sigma} \wr \mathbb{Z}$  with an undecidable
membership problem.
\end{theorem}
Finally, we can establish the main result of this section.

\begin{theorem}
There is a
fixed finitely generated submonoid $M$ of the wreath product
$\mathbb{Z} \wr \mathbb{Z}$  with an undecidable
membership problem.
\end{theorem}

\begin{proof}
By Theorem~\ref{thm:undec1} it suffices to reduce the submonoid membership problem of
$\mathbb{Z}^{\Sigma} \wr \mathbb{Z}$ to the submonoid membership problem of
$\mathbb{Z} \wr \mathbb{Z}$.  If $m=|\Sigma|$, then Proposition~\ref{finiteindex} shows that $\mathbb{Z}^{\Sigma} \wr \mathbb{Z}\cong \mathbb Z^m\wr m\mathbb Z$ is isomorphic to a subgroup of index $m$ in $\mathbb Z\wr \mathbb Z$.  So if $\mathbb Z\wr \mathbb Z$ had decidable submonoid membership for each finitely generated submonoid, then the same would be true of $\mathbb{Z}^{\Sigma} \wr \mathbb{Z}$.
\qed
\end{proof}

We remark that, together with the undecidability of the rational subset
membership problem for groups $H\wr
(\mathbb{Z}\times\mathbb{Z})$ for non-trivial $H$~\cite{LohSt09tocs}, our results imply the
following: For finitely generated non-trivial abelian groups $G$ and $H$, the wreath
product $H\wr G$ has a decidable rational subset membership problem if and only if 
(i) $G$ is finite\footnote{If $G$ has size $m$, then by Proposition~\ref{finiteindex},
$H^m \cong H^m \wr 1$ is isomorphic to a subgroup of index $m$ in $H \wr G$.
Since $H^m$ is finitely generated abelian, decidability of the rational subset membership 
problem of $H \wr G$ follows from the fact that decidability is preserved by finite extensions
\cite{Gru99,KaSiSt06}.} or 
(ii) ($G$ has rank 1 and $H$ is finite).
Furthermore, for virtually free groups $G$ and $H$, the rational subset
membership problem is decidable for $H\wr G$ if and only if (i) $G$ is trivial or
(ii) $H$ is finite, or (iii) ($G$ is finite and $H$ is virtually abelian).

By \cite{Clea06}, the wreath product $\mathbb{Z} \wr \mathbb{Z}$ is a subgroup of Thompson's group $F$
as well as of Baumslag's finitely presented metabelian group 
$\langle a, s, t \mid [s,t]=[a^t, a]=1, a^s = aa^t \rangle$. Hence, we get:

\begin{corollary}
Thompson's group $F$ as well as Baumslag's finitely presented metabelian group 
both contain finitely generated submonoids with an undecidable membership problem.
\end{corollary}

\section{Open problems}

As already mentioned in the introduction, we conjecture that
the rational subset membership problem for a wreath product
$H \wr G$ with $H$ non-trivial and $G$ not virtually free is
undecidable. Another interesting
case, which is not resolved by our results, concerns wreath products
$G \wr V$ with $V$ virtually free and $G$ a finitely generated
infinite torsion group.
Finally, all these questions can be also asked for the
submonoid membership problem. We do not know any example of
a group with decidable submonoid membership problem but undecidable
rational subset membership problem. If such a group exists, it must be
one-ended \cite{LohSt10}.


\def\cprime{$'$}

\end{document}